\pdfoutput=1
\RequirePackage{ifpdf}
\ifpdf 
\documentclass[pdftex]{sigma}
\else
\documentclass{sigma}
\fi

\usepackage{dsfont}
\usepackage[final]{listings}
\usepackage{tikz}
\usetikzlibrary{matrix,arrows}

\numberwithin{equation}{section}

\newtheorem{Theorem}{Theorem}[section]
\newtheorem{Corollary}[Theorem]{Corollary}
\newtheorem{Proposition}[Theorem]{Proposition}
 { \theoremstyle{definition}
\newtheorem{Definition}[Theorem]{Definition}}

\newcommand{\bbA}{{\mathbb A}}
\newcommand{\bbQ}{{\mathbb Q}}
\newcommand{\bbT}{{\mathbb T}}
\newcommand{\bbZ}{{\mathbb Z}}
\newcommand{\frake}{{\mathfrak{e}}}
\newcommand{\frakg}{{\mathfrak{g}}}
\newcommand{\frakk}{{\mathfrak{k}}}
\newcommand{\frakm}{{\mathfrak{m}}}
\newcommand{\rmf}{{\mathrm{f}}}
\newcommand{\rmt}{{\mathrm{t}}}
\newcommand{\rmA}{{\mathrm{A}}}
\newcommand{\rmC}{{\mathrm{C}}}
\newcommand{\rmD}{{\mathrm{D}}}
\newcommand{\rmE}{{\mathrm{E}}}
\newcommand{\rmG}{{\mathrm{G}}}
\newcommand{\rmK}{{\mathrm{K}}}
\newcommand{\rmL}{{\mathrm{L}}}
\newcommand{\rmM}{{\mathrm{M}}}
\newcommand{\rmR}{{\mathrm{R}}}
\newcommand{\rmS}{{\mathrm{S}}}
\newcommand{\rmT}{{\mathrm{T}}}
\newcommand{\ov}{\overline}
\newcommand{\ul}{\underline}
\newcommand*{\longhookrightarrow}{ \lhook\joinrel\relbar\joinrel\rightarrow}
\newcommand*{\longtwoheadrightarrow}{\relbar\joinrel\twoheadrightarrow}
\newcommand{\ra}{\rightarrow}
\newcommand{\hra}{\hookrightarrow}

\newcommand{\lra}{\longrightarrow}
\newcommand{\lhra}{\longhookrightarrow}
\newcommand{\lthra}{\longtwoheadrightarrow}
\newcommand{\mto}{\mapsto}
\newcommand{\lmto}{\longmapsto}

\newcommand{\ZZ}{{\mathbb{Z}}}
\newcommand{\RR}{{\mathbb{R}}}
\newcommand{\CC}{{\mathbb{C}}}

\renewcommand{\Im}{\operatorname{\mathfrak{Im}}}
\newcommand{\Hom}{\operatorname{\mathrm{Hom}}}

\newcommand{\Mat}[1]{{\mathrm{Mat}_{#1}}}
\newcommand{\GL}[1]{{\mathrm{GL}_{#1}}}
\newcommand{\PGL}[1]{{\mathrm{PGL}_{#1}}}
\newcommand{\SL}[1]{{\mathrm{SL}_{#1}}}
\newcommand{\Sp}[1]{{\mathrm{Sp}_{#1}}}
\newcommand{\Orth}[1]{{\mathrm{O}_{#1}}}

\newcommand{\U}[1]{{\mathrm{U}_{#1}}}

\newcommand{\T}{{\rmt}}
\newcommand{\rT}{{\,{}^\T\!}}
\newcommand{\sym}{{\mathrm{sym}}}
\newcommand{\bbone}{{\mathds{1}}}

\newcommand{\lspan}{\operatorname{span}}
\newcommand{\HS}{\mathbb{H}}
\newcommand{\Rep}{\operatorname{Rep}}
\newcommand{\gKmodule}{$(\frakg,\rmK)$-module}
\newcommand{\gKmodules}{$(\frakg,\rmK)$-modules}
\newcommand{\HSg}{{\HS^{(g)}}}
\newcommand{\PSL}[1]{{\mathrm{PSL}_{#1}}}
\newcommand{\GSp}[1]{{\mathrm{GSp}_{#1}}}
\newcommand{\PGSp}[1]{{\mathrm{PGSp}_{#1}}}
\newcommand{\Ga}[1]{{\Gamma^{(#1)}}}
\newcommand{\Af}{{{\bbA_\rmf}}}
\newcommand{\Qp}{{{\bbQ_p}}}
\newcommand{\Zp}{{\bbZ_p}}
\newcommand{\fraksp}[1]{{\mathfrak{sp}_{#1}}}
\newcommand{\KgRR}{{K^{(g)}_\infty}}
\newcommand{\KgQp}{{K^{(g)}_p}}
\newcommand{\KgAf}{{K^{(g)}_\rmf}}
\newcommand{\ovRep}{{\ov{\mathrm{Rep}}}}
\newcommand{\IrrRep}{{\ov{\mathrm{IrrRep}}}}
\newcommand{\IrrRepAf}{{\IrrRep\big( \KgAf \big)}}
\newcommand{\IrrRepRR}{{\IrrRep\big( \KgRR \big)}}
\newcommand{\Mg}{{\rmM^{(g)}}}
\newcommand{\aMg}[1]{{\sideset{^{[#1]}}{^{(g)}}{\mathop{\rmM}}}}
\newcommand{\Mgall}{{\Mg( \bullet \otimes \bullet)}}
\newcommand{\aMgall}{{\aMg{\bullet}( \bullet \otimes \bullet)}}

\newcommand{\ulcdot}{{\mathop{\ul{\,\cdot\,}}}}
\newcommand{\ulotimes}{{\mathop{\ul{\otimes}}}}
\newcommand{\Ind}{{\mathrm{Ind}}}

\begin{document}

\allowdisplaybreaks

\renewcommand{\thefootnote}{}

\renewcommand{\PaperNumber}{108}

\FirstPageHeading

\ShortArticleName{Hyper-Algebras of Vector-Valued Modular Forms}

\ArticleName{Hyper-Algebras of Vector-Valued Modular Forms\footnote{This paper is a~contribution to the Special Issue on Modular Forms and String Theory in honor of Noriko Yui. The full collection is available at \href{http://www.emis.de/journals/SIGMA/modular-forms.html}{http://www.emis.de/journals/SIGMA/modular-forms.html}}}

\Author{Martin RAUM}

\AuthorNameForHeading{M.~Raum}

\Address{Chalmers tekniska h\"ogskola och G\"oteborgs Universitet, \\ Institutionen f\"or Matematiska vetenskaper, SE-412 96 G\"oteborg, Sweden}
\Email{\href{mailto:martin@raum-brothers.eu}{martin@raum-brothers.eu}}
\URLaddress{\url{http://raum-brothers.eu/martin/}}

\ArticleDates{Received May 07, 2018, in final form September 30, 2018; Published online October 04, 2018}

\Abstract{We define graded hyper-algebras of vector-valued Siegel modular forms, which allow us to study tensor products of the latter. We also define vector-valued Hecke operators for Siegel modular forms at all places of~${\mathbb Q}$, acting on these hyper-algebras. These definitions bridge the classical and representation theoretic approach to Siegel modular forms. Combining both the product structure and the action of Hecke operators, we prove in the case of elliptic modular forms that all cusp forms of sufficiently large weight can be obtained from products involving only two fixed Eisenstein series. As a byproduct, we obtain inclusions of cuspidal automorphic representations into the tensor product of global principal series.}


\Classification{11F40; 11F60; 11F70}

\renewcommand{\thefootnote}{\arabic{footnote}}
\setcounter{footnote}{0}

\section{Introduction}

Products of scalar-valued modular forms can be used to construct new ones, and in special situations they reveal deep arithmetic information. For example, Rankin showed in~\cite{rankin-1952} that the Petersson scalar product of an elliptic cusp form~$f$ and the product $E_k E_l$ of two Eisenstein series, can be expressed in terms of special $L$-values attached to~$f$. This very relation reappeared in~\cite{kohnen-zagier-1984}, where Kohnen and Zagier define a ${\mathbb Q}$-structure on elliptic modular forms that is different from the one originating in Fourier expansions. The product of scalar-valued modular forms naturally leads to the definition of the graded algebra $\rmM(\bullet)$ of modular forms.

Vector-valued modular forms in the elliptic case can be associated with any representation~$\rho$ of the modular group $\SL{2}(\ZZ)$. As opposed to scalar-valued ones, their product structure has mostly been neglected. Instead, one considers graded modules $\rmM(\bullet \otimes \rho )$ over~$\rmM(\bullet)$. For example, Marks and Mason proved in~\cite{marks-mason-2010} that $\rmM(\bullet \otimes \rho )$ is free over $\rmM(\bullet)$ and that its rank relates directly to the eigenvalues of $\rho\big({-}I^{(2)})$ for $I^{(2)}$ the $2 \times 2$ identity matrix.

It is possible to define tensor products of vector-valued modular forms by means of the tensor product of smooth functions from the Poincar\'e upper half plane $\HS = \{ \tau \in \CC\colon \Im \tau > 0 \}$ to the representation spaces $V(\rho)$ and $V(\rho')$ of $\rho$ and $\rho'$
\begin{gather*}
 \otimes\colon \
 \rmC^\infty \big( \HS \ra V(\rho) \big)
 \times
 \rmC^\infty \big( \HS \ra V(\rho') \big)
\lra
 \rmC^\infty \big( \HS \ra V(\rho \otimes \rho') \big),\\
\hphantom{\otimes\colon}{} \ (f \otimes g)(\tau) \lmto f(\tau) \otimes g(\tau).
\end{gather*}
Tensor products of vector-valued modular forms have seldom been studied. In~\cite{raum-2017}, we employed them to express elliptic cusp forms of any level as tensor products of at most two Eisenstein series. Tensor products can be conveniently subsumed under the concept of hyper-algebras. Our description of tensor products of and differential operators on almost holomorphic Siegel modular forms that we gave in~\cite{klemm-poretschkin-schimannek-raum-2015}, for example, employed hyper-algebras to classify almost holomorphic Siegel modular forms. In this work, we review the concept of hyper-algebras in the context of modular forms.

Hyper-algebras mimic algebras but they have multivalued multiplication. In particular, they are natural analogues of hyper-groups~\cite{wall-1937}. Note that we suppress Wall's assumption of dimensionality of hyper-groups, which does not hold for examples that we treat. Hyper-groups arise naturally in the context of double cosets in group theory. The definition of a group~$G$ includes the product~$a \cdot b \in G$ of $a, b \in G$. Given a sufficiently nice subgroup~$H$~-- that is, if $(G,H)$ is a~Gelfand pair~-- it is standard in modular forms to define a (commutative) Hecke algebra. Specifically, Hecke algebras assign meaning to the product of double cosets $H a H \cdot H b H$. For instance, if $G = \GSp{g}({\mathbb Q})$ and $K = \GSp{g}(\ZZ)$, then the Hecke algebra with integral coefficients is free as a $\ZZ$-module and has basis $H c H$ for certain diagonal $c \in G$. Every product of double cosets can be written as a sum $\sum_{c \in G} m_c H c H$, where $0 \le m_c \in \ZZ$ and $m_c \ne 0$ for only finitely many~$c$. The corresponding notion of hyper-groups, to the author's knowledge, has not yet appeared in the context of modular forms, but is standard in other areas of mathematics. Rephrasing the product in Hecke algebras, we arrive at a notion of multivalued multiplication of double cosets. One can decompose $H a H = \bigcup_{c \in \rmL(H,a)} H c$ and $H b H = \bigcup_{c \in \rmL(H,b)} H c$ as a disjoint union of left cosets with coset representatives $\rmL(H,a), \rmL(H,b) \subseteq G$. By definition of the multiplication in Hecke algebras, the multiset $\{\!\{ H ( d_a \cdot d_b )\colon d_a \in \rmL(H,a), d_b \in \rmL(H,b) \}\!\}$ has a decomposition into the disjoint union of $\{\!\{ H d_c \colon d_c \in \rmL(H,c) \}\!\}$ where each $\rmL(H,c)$ occurs with multiplicity $m_c$ that occurred before. A more detailed, explicit explanation can be found in~\cite{krieg-1990}. Summarizing, we obtain a multiplication of double cosets that takes values in multisets of double cosets:
\begin{gather*}
 H \backslash G \slash H \times H \backslash G \slash H\lra \mathrm{Multiset}( H \backslash G \slash H ).
\end{gather*}
In this way, we obtain a commutative hyper-group. Note that every group can be naturally viewed as a hyper-group. Vice versa, a hyper-group whose multiplication takes values in multisets of size~$1$ is a group.

It is possible to extend the definition of hyper-groups to hyper-algebras. Addition is axioma\-tized as in the case of algebras. Multiplication takes values in the set of all subspaces and compatibility relations that mimic those for algebras are imposed on it. We give a precise definition of hyper-algebras in Section~\ref{sec:hyper-algebras}. Given an algebra~$A$ we obtain a hyper-algebra by assigning to two elements the module spanned by their product. In general, it is not possible to recover the algebra from this, because the hyper-algebra product $a \ul{\,\cdot\,} b$ associated with $a, b \in A$ and $(r a) \ul{\,\cdot\,} (r' b)$ is the same for all units $r$, $r'$ in the base ring~$R$.

Siegel modular forms are assigned to a weight~$\sigma$ and a type~$\rho$, which are representations of $\GL{g}(\CC)$ and $\Sp{g}(\ZZ)$, respectively. Given Siegel modular forms $f$ and $g$ of weights and types $\sigma_f$, $\rho_f$, $\sigma_g$, and $\rho_g$, then their tensor product~$f \otimes g$ has weight~$\sigma_f \otimes \sigma_g$ and type $\rho_f \otimes \rho_g$. Even if the weights and types of $f$ and $g$ are irreducible, their tensor products can be reducible. On the other hand, to avoid redundancies, graded modules of Siegel modular forms are preferably build from irreducible weights and types. As a consequence, tensor products do not yield an algebra structure on such modules.

Since weights and types behave similarly with respect to the construction that we discuss now, we focus on the former. To remedy the lack of algebra structures on the graded module of Siegel modular forms, we have previously suggested to attach to $f$ and $g$ the span
\begin{gather*}
 f \,\ulotimes\, g= \lspan \big\{ \phi \circ ( f \otimes g ) \colon \phi \in \Hom( \sigma_f \otimes \sigma_g, \sigma )\big\}\subset \bigoplus_\sigma \Mg (\sigma),
\end{gather*}
where $\sigma$ runs through a fixed set of representatives of isomorphism classes of irreducible weights, and $\Mg(\sigma)$ denotes the corresponding space of Siegel modular forms. In~\cite{klemm-poretschkin-schimannek-raum-2015}, we identified the resulting structure as a hyper-algebra. It is a special case of hyper-algebras of Siegel modular forms, defined in Section~\ref{ssec:modular-forms-hyper-algebra}.

In~\cite{klemm-poretschkin-schimannek-raum-2015}, we also defined the concept of differential hyper-algebras to subsume the action of covariant differential operators. In Section~\ref{ssec:hyper-algebras:differential-operators}, we revisit this definition and reinterpret it in terms of a Hecke operator~$\rmT_\infty$ at the infinite place of~${\mathbb Q}$. Returning to the case of general weights and types, we also study vector-valued Hecke operators~$\rmT_v$ at all places~$v$ of~${\mathbb Q}$. If $v = p$ is a prime, then $\rmT_p$ generalizes classical Hecke operators for Siegel modular forms. It extends the vector-valued Hecke operators for elliptic modular forms that we have already defined in~\cite{raum-2017}. The polynomial algebra~$\bbT$ generated by formal elements~$\rmT_v$ acts by means of these Hecke operators on the hyper-algebra of Siegel modular forms. Specifically, at the infinite place, the Hecke operator acts as a hyper-derivation. At the finite places, Hecke operators respect the hyper-product structure. A precise description can be found in Section~\ref{ssec:hyper-algebras:hecke-algebra}.

Section~\ref{sec:structure-results} contains a more detailed study of tensor products and the action of~$\bbT$ on modular forms is studied in more detail in the genus~$1$ case. Let $\rmS^{(1)}( k)$ be the space of genus~$1$ cusp forms of weight~$k$ and level~$1$. Write $E^{(1)}(l)$ for the level~$1$ Eisenstein series of weight~$l$. The following is a slight modification of the main result in~\cite{raum-2017}.
\begin{Theorem}\label{thm:main-theorem:classical}
Let $l,l' \ge 4$ be even integers and $\rho$ a finite-dimensional representation of $\SL{2}(\ZZ)$ whose kernel is a congruence subgroup. Then for every $k \ge l + l'$, we have
\begin{gather}\label{eq:thm:main-theorem:classical}
 \rmS^{(1)}( k \otimes \rho )\subset \bbT \big( \bbT E^{(1)}(l) \,\ulotimes\, \bbT E^{(1)}(l') \big).
\end{gather}
\end{Theorem}

The right-hand side of~\eqref{eq:thm:main-theorem:classical} can be efficiently computed and its left-hand side is genuinely interesting. We argue that the proof can serve as a blueprint to a whole family of analogous statements. Thus Theorem~\ref{thm:main-theorem:classical} sets the path towards a general method of computing Siegel modular forms. Section~\ref{sec:structure-results} contains a specific conjecture in the case of genus~$2$.

Next, we discuss a connection between Theorem~\ref{thm:main-theorem:classical} and automorphic representations for $\PGL{2}$. To this end, we restrict to representations~$\rho$ whose kernel is a congruence subgroup. Vector-valued Hecke operators defined in this paper can be related to adelic automorphic representations. Covariant differential operators correspond to the action of an appropriate Lie algebra on Harish-Chandra modules. Hecke operators at finite places arise directly from a representation of a group over~$\Qp$. It is natural to ask how much Theorem~\ref{thm:main-theorem:classical} relates to results from automorphic representation theory. We make such a relation precise, and determine some constituents in the tensor product of principal series. Specifically, we consider automorphic representations for~$\PGL{2}$. We let $\varpi(k-1, \bbone)$ be the principal series that is unramified at the finite places and has Harish-Chandra parameter $k - 1$ at infinity. Let $\varpi_\infty(k-1)$ be the discrete series with Harish-Chandra parameter~$k-1$.
\begin{Theorem}\label{thm:main-theorem:automorphic}Given even integers $l,l' \ge 4$, then
\begin{gather*}
 \varpi\lhra \varpi(l-1,\bbone) \otimes \varpi(l'-1, \bbone)
\end{gather*}
for every cuspidal automorphic representation $\varpi = \varpi_\infty(k-1) \otimes \varpi_\rmf$ with $k \ge l + l'$ and $\varpi_\rmf$ a~representation of~$\PSL{2}(\bbA_\rmf)$.
\end{Theorem}
When phrased in this language, similarity to the Gross--Prasad conjectures for the inclusion $\Orth{1,2} \cong \SL{2} \hra \SL{2} \times \SL{2} \cong \Orth{2,2}$ becomes apparent.

In Section~\ref{sec:modular-forms}, we collect preliminaries on Siegel modular forms. In Section~\ref{sec:hyper-algebras}, we define hyper-algebras of Siegel modular forms. In Section~\ref{sec:structure-results}, we show that products of certain Eisenstein series yield cusp forms. In Section~\ref{sec:automorphic-representations}, we give an interpretation of our results in terms of adelic automorphic representations.

\section{Modular forms}\label{sec:modular-forms}

\subsection{The classical setup}

The Siegel upper half space $\big\{ \tau \in \Mat{g}(\CC)\colon \rT\tau = \tau,\, \Im\tau > 0 \big\}$ of genus~$g$ is denoted by~$\HS^{(g)}$. It carries an action of the real symplectic group
\begin{gather*}
 \Sp{g}(\RR)= \big\{ \gamma \in \Mat{2g}(\RR) \colon \rT\gamma J^{(g)} \gamma = J^{(g)} \big\},\qquad J^{(g)} = \left(\begin{smallmatrix}
 0 & -I^{(g)} \\
 I^{(g)} & 0
 \end{smallmatrix}\right),
\end{gather*}
where $I^{(g)}$ is the $g \times g$ identity matrix. This action is explicitly given by $\gamma \tau = (a \tau + b) (c \tau + d)^{-1}$, where we employ the block decomposition of $\gamma = \left(\begin{smallmatrix} a & b \\ c & d \end{smallmatrix}\right)$. The subgroup $\Ga{g} = \Sp{g}(\ZZ)$ of symplectic transformation matrices with integral entries is called the Siegel modular group of genus~$g$.

Generally, a complex representation~$\sigma$ of~$\GL{g}(\CC)$ is called a weight. For the purpose of this paper, we restrict to representations that factor through $\GL{g}(\CC) \slash \big\{ {\pm} I^{(g)} \big\}$. A finite-dimensional representation~$\rho$ of $\Sp{g}(\ZZ)$ is called a type. Throughout this note, we focus on types whose kernel is a congruence subgroup that contains $-I^{(g)}$. The representation spaces of $\sigma$ and $\rho$ are denoted by~$V(\sigma)$ and~$V(\rho)$. A weight and a type together determine a slash action
\begin{gather*}
 \big( f \big|_{\sigma,\rho} \gamma \big) (\tau)= \sigma(c \tau + d)^{-1} \rho(\gamma)^{-1} f(\gamma \tau)
\end{gather*}
on functions $f\colon \HS^{(g)} \ra V(\sigma) \otimes V(\rho)$.

The space of genus~$g$ Siegel modular forms of weight~$\sigma$ and type~$\rho$ is defined as the space of holomorphic functions~$f \colon \HS^{(g)} \ra V(\sigma) \otimes V(\rho) $ such that $f |_{\sigma,\rho} \gamma = f$ for all $\gamma \in \Ga{g}$ and which satisfy $f(x + {\rm i}y) = O(1)$ as $y \ra \infty$, if $g=1$. We write
\begin{gather*}
 \Mg( \sigma \otimes \rho)= \rmE^{(g)}(\sigma \otimes \rho ) \oplus \rmS^{(g)}( \sigma \otimes \rho)
\end{gather*}
for this space, the space of Eisenstein series, and the space of cusp forms, respectively.

\subsection{The symplectic group and its Lie algebra}\label{ssec:modular-forms:symplectic-group}

We write $\rmG = \PGSp{g}$ for the ${\mathbb Q}$-split algebraic group of projective symplectic similitudes. Throughout we work with the model
\begin{gather*}
 \PGSp{g}(\ZZ)= \big\{ \gamma \in \Mat{2g}(\ZZ) \colon \rT\gamma J^{(g)} \gamma = s(\gamma) J^{(g)},\, s(\gamma) \in \ZZ \big\}\big\slash \big\{ s I^{(g)} \colon s \in \ZZ \big\}.
\end{gather*}
Compact subgroups $\KgRR$ and $\KgQp$ of $\rmG(\RR)$ and $\rmG(\Qp)$ are $\U{g}(\RR) \slash \big\{{\pm}I^{(g)}\big\}$ and $\PGSp{g}(\Zp)$, where the former is embedded into $\rmG(\RR)$ by $a {\rm i} + b \mto \left(\begin{smallmatrix}a & -b \\ b & a \end{smallmatrix}\right)$. Weights correspond to complex representations of $\KgRR$ by means of the restriction along $\U{g}(\RR) \slash \big\{{\pm}I^{(g)}\big\} \ra \GL{g}(\CC) \slash \big\{{\pm}I^{(g)}\big\}$.

We let $\ovRep\big(\KgRR\big)$ be a fixed set of representatives of isomorphism classes of finite-dimensional, complex representations of $\KgRR$. For all $g$, the representation $\det^k \colon \U{g}(\RR) \ra \CC$, $k \mto \det(k)$ with $k \in \ZZ$, $2 \,|\, g k$ yields an irreducible representation of~$\KgRR$ via the above isomorphism of $\KgRR$ and $\U{g}(\RR) \slash \big\{{\pm}I^{(g)}\big\}$. In the context of spaces of modular forms, we will occasionally write $k$ instead of $\det^k$ to match more closely the classical notation. For example,
\begin{gather*}
 \Mg( k \otimes \rho) = \Mg\big({\det}^k \otimes \rho \big).
\end{gather*}
We denote the $l$-th symmetric power of the standard representation by $\sym^l$. The dual of a~representation is indicated by a superscript $\vee$: $\sigma^\vee$ is the dual of $\sigma$.

We let $\ovRep\big(\KgQp\big)$ be a fixed set of representatives of isomorphism classes of finite-dimensional, complex representations of $\KgQp$. To simplify notation, we let $\ovRep\big(\KgAf\big)$ be the set of restricted tensor products of the representations in $\ovRep\big(\KgQp\big)$. That is, it consists of exterior tensor products of representations of $\KgQp$ which are trivial for all but finitely many~$p$. There is a~correspondence between $\ovRep\big(\KgAf\big)$ and finite-dimensional, complex representations of $\Ga{g}$ with a congruence subgroup and the matrix $-I_{2g}$ in their kernel.

Subsets of irreducible representations are denoted by
\begin{gather*}
 \IrrRep\big(\KgRR\big)\subset \ovRep\big(\KgRR\big)\qquad\text{and}\qquad \IrrRep\big(\KgAf\big)\subset \ovRep\big(\KgAf\big).
\end{gather*}

\subsection{Covariant differential operators}

We call a differential operator~$\rmD$ covariant from $\big|_{\sigma,\rho}$ to $\big|_{\sigma',\rho}$ if for all $g \in \rmG(\RR)$ and all smooth functions $f \colon \HS^{(g)} \ra V(\sigma) \otimes V(\rho)$ , we have
\begin{gather*}
 \rmD\big( f \big|_{\sigma,\rho} g \big)= \rmD(f) \big|_{\sigma',\rho} g.
\end{gather*}
If $\rmD$ is covariant for the trivial type then it yields a differential operator that is covariant for all types.

A theorem of Helgason~\cite{helgason-1962}, allows us to classify order~$1$ differential operators. For each~$\sigma$ there is a lowering operator~$\rmL = \rmL_\sigma$ that is covariant from $\sigma$ to $\rmL \sigma = \sym^{2 \vee} \sigma$ and a raising operator~$\rmR = \rmR_\sigma$ that is covariant from~$\sigma$ to $\rmR \sigma = \sym^2 \sigma$. We only fix a normalization in the case~$g = 1$, setting
\begin{gather*}
 \rmL= -2 {\rm i} y^2 \partial_{\ov\tau} \qquad\text{and}\qquad \rmR = 2 {\rm i} \partial_\tau + k y^{-1}.
\end{gather*}
In~\cite{klemm-poretschkin-schimannek-raum-2015}, we gave explicit expressions for all~$g$. The normalization employed there, however, does not coincide with the one we choose here.

\looseness=-1 Covariant differential operators allow us to define almost holomorphic Siegel modular forms. A~function $f \colon \HS^{(g)} \ra V(\sigma) \otimes V(\rho)$ that vanishes under the $(d+1)$-th tensor power of the lowering operator, i.e., $\rmL^d f = 0$, that is invariant of weight~$\sigma$ and type~$\rho$, i.e., $f \big|_{\sigma,\rho} \gamma = f$ for all $\gamma \in \Ga{g}$, and that satisfies $f(\tau) = O(1)$ as $y \ra \infty$, if $g = 1$, is called an almost holomorphic Siegel modular form of weight~$\sigma$, type~$\rho$, and depth~$d$. The space of such functions is denoted by
\begin{gather*}
 \aMg{d} ( \sigma \otimes \rho ).
\end{gather*}

\subsection{Hecke operators}\label{ssec:modular-forms:hecke-operators}

We extend the definition of vector-valued Hecke operators in~\cite{raum-2017} for elliptic modular forms to Siegel modular forms. Many of the proofs in~op.\ cit.\ apply to the general case word by word. For this reason, we skip several arguments in this subsection.

\subsubsection*{Hecke operators on representations}\label{ssec:hecke-operators-on-representations}

For a positive integer~$M$, we let
\begin{gather*}
 \Delta_M= \Big\{
 \left(\begin{smallmatrix} a & b \\ 0 & d \end{smallmatrix}\right)
 \in \GSp{g}({\mathbb Q})\colon
 \rT a d = M I^{(g)},\,
 a, b, d \in \Mat{g}(\ZZ),\,
 d\ \text{upper triangular},\\
 \hphantom{\Delta_M= \Big\{}{}
 \forall\, i < j \colon 0 \le d_{i,j} < d_{j,j},\, \forall\, i \colon, 0 \le b_{i,j} < d_{j,j} \Big\},
\end{gather*}
where $\GSp{g}$ is the group of symplectic similitude transformations. We have a right action of~$\Sp{g}(\ZZ)$ on $\Delta_M$ defined by $(m, \gamma) \mto \ov{m \gamma}$ with $\gamma' \ov{m \gamma} = m \gamma$ for some $\gamma' \in \Sp{g}(\ZZ)$. This, in particular, defines a cocycle $I_m(\gamma) = \gamma'$; That is, we have $I_{m}(\gamma_1 \gamma_2) = I_m(\gamma_1) I_{m \gamma_1}(\gamma_2)$.

We denote the natural basis of~$\CC[\Delta_M]$ by~$\frake_m$, $m \in \Delta_M$. To every type~$\rho$ we associated the type $T_M \rho$ defined by
\begin{gather*}
 V(T_M \rho):= V(\rho) \otimes \CC[ \Delta_M ]\qquad\text{and}\qquad (T_M \rho)(\gamma) (v \otimes \frake_{m}):= \rho\big( I_m^{-1}\big(\gamma^{-1}\big) \big) (v) \otimes \frake_{m \gamma^{-1}}.
\end{gather*}
The cocycle property of $I_m(\gamma)$ implies that it is a representation of~$\Sp{g}(\ZZ)$. Given a scalar product~$\langle\,\cdot\,,\,\cdot\,\rangle_\rho$ on $V(\rho)$, we obtain one on $V(\rmT_M \rho)$ by
\begin{gather*}
\langle v \otimes \frake_{m_v}, w \otimes \frake_{m_w} \rangle= \begin{cases}
 \langle v, w \rangle_\rho, & \text{if $m_v = m_w$}, \\
 0, & \text{otherwise}.
 \end{cases}
\end{gather*}

Hecke operators on representations are compatible with homomorphisms between types and with tensor products. Specifically, the following are homomorphism of $\Sp{g}(\ZZ)$-representations:
\begin{gather}
 T_M \phi \colon \ T_M \rho \lra T_M \rho',\qquad v \otimes \frake_m \lmto \phi(v) \otimes \frake_m,\nonumber\\
\hspace*{17.5mm} \bbone \lhra T_M \bbone,\hspace*{17mm}
 c \lmto c \sum_{m \in \Delta_M} \frake_m,\nonumber\\
 (T_M \rho) \otimes (T_M \sigma)\lthra T_M (\rho \otimes \sigma),\nonumber\\
\label{eq:prop:hecke-operator-comonoidal-on-reps:comonoidal-coherence}
 (v \otimes \frake_m) \otimes (w \otimes \frake_{m'})
\lmto
 \begin{cases}
 (v \otimes w) \otimes \frake_m , & \text{if $m = m'$}, \\
 0 , & \text{otherwise},
 \end{cases}
\end{gather}
where $\phi \in \Hom( \rho, \rho' )$. If $\psi \colon \rho \ra \rho''$ is a further homomorphism, then $T_M (\psi \circ \phi) = (T_M \psi) \circ (T_M \phi)$. For later reference, we denote the morphism in~\eqref{eq:prop:hecke-operator-comonoidal-on-reps:comonoidal-coherence} by $\pi_{M, \rho, \rho'}$.

\subsubsection*{Hecke operators on modular forms}\label{ssec:hecke-operators-on-modular-forms}

For $m = \left(\begin{smallmatrix} a & b \\ 0 & d \end{smallmatrix}\right) \in \GSp{g}(\RR)$ of similitude~$M$, and for $f \colon \HS \ra V(\sigma)$, we define
\begin{gather*}
 \big( f \big|_\sigma m \big) (\tau)= \sigma\big( d \slash \sqrt{M} \big)^{-1} f\big((a\tau + b) d^{-1}\big).
\end{gather*}

Fixing a type~$\rho$ and a positive integer~$M$, we define the vector-valued Hecke operator $\rmT_M$ acting on $f \in \Mg\big( \sigma \otimes \rho \big)$ by
\begin{gather*}
 \big( T_Mf \big)(\tau)= \sum_{m \in \Delta_M} \big( f \big|_\sigma m \big)(\tau) \otimes \frake_m\in \Mg\big( \sigma \otimes \rmT_M \rho \big).
\end{gather*}
The above homomorphism $\pi_{M,\rho, \rho'}$ is compatible with Hecke operators. Specifically, we have
\begin{gather*}
 \pi_{M, \rho, \rho'} \big( (\rmT_M f) \otimes (\rmT_M g) \big)= \rmT_M \big( f \otimes g \big).
\end{gather*}

\section{Hyper-algebras}\label{sec:hyper-algebras}

We start with the formal definition of hyper-groups, which is the blueprint to our Defini\-tion~\ref{def:hyper-algebra} of hyper-algebras. Given a set~$S$, let $\mathrm{Multiset}(S) \cong \{ f \colon S \ra \ZZ_{\ge 0} \}$ be the set of all multisubsets of~$S$. Multisets are throughout denoted by double curly brackets $\{\!\{ \cdots \}\!\}$. The union of multisets corresponds to the sum of functions $S \ra \ZZ_{\ge 0}$. As a shorthand notation, for the binary operator appearing in the next definition, we set $a \ulcdot \{\!\{ b_i \}\!\}_i = \bigcup_{b_i} a \ulcdot b_i$ and similarly $\{\!\{ a_i \}\!\}_i \ulcdot b = \bigcup_{a_i} a_i \ulcdot b$.
\begin{Definition}
A pair $(G,\ulcdot)$ of a set~$G$ and a binary operator~$\ulcdot \colon G \times G \ra \mathrm{Multiset}(G)$ is called a hyper-group if the following axioms are satisfied:
\begin{enumerate}\itemsep=0pt
\item[(i)] (Finite image) For any $a, b \in G$, the multiset $a \ulcdot b$ is finite.
\item[(ii)] (Associativity) For any $a,b,c \in G$, we have $(a \ulcdot b) \ulcdot c = a \ulcdot (b \ulcdot c)$, that is,
\begin{gather*}
 \bigcup_{h \in a \ulcdot b} h \ulcdot c= \bigcup_{h \in b \ulcdot c} a \ulcdot h.
\end{gather*}
\item[(iii)] (Identity) There is an element $e \in G$ such that for every~$a \in G$ we have $a \in a \ulcdot e$ and $a \in e \ulcdot a$.
\item[(iv)] (Inverse) For every element~$a \in G$ there is element~$a^{-1} \in G$ such that $e \in a \ulcdot a^{-1}$ and $e \in a^{-1} \ulcdot a$.
\end{enumerate}
\end{Definition}

Extending this notion to hyper-algebras is straightforward. However, the reader should be warned that the word ``hyper-algebra'' is used in the context of algebraic groups~\cite{sullivan-1978}, too, and these two notions should not be confused.

Given a commutative ring~$R$ (with identity), and an $R$-module~$M$, we let $\mathrm{SubMod}_R(M)$ be the set of all $R$-submodules of~$M$. In analogy with the definition of multisets, we extend the binary operator~$\ulcdot$ in the following definition to submodules by $a \ulcdot N = \bigcup_{b \in N} a \ulcdot b$ and $N \ulcdot b = \bigcup_{a \in N} a \ulcdot b$ for $N \in \mathrm{SubMod}_R(M)$.
\begin{Definition}\label{def:hyper-algebra}
Let $R$ be a ring. A triple $(A,+,\,\ulcdot\,)$ with $(A,+)$ an $R$-module and binary operator
\begin{gather*}
 \ulcdot \colon \ A \times A \lra \mathrm{SubMod}_R(A)
\end{gather*}
is called a hyper-algebra (with identity) if
\begin{enumerate}\itemsep=0pt
\item[(i)] (Linearity) Given $a, b \in A$ and $r \in R$, we have $(r a) \ulcdot b = r \big( a \ulcdot b \big) = a \ulcdot (r b)$.

\item[(ii)] (Identity) There is $e \in A$ such that we have $e \ulcdot a = a \ulcdot e = \lspan_R a$ for all $a \in A$.

\item[(iii)] (Associativity) Given $a,b,c \in A$, we have $a \ulcdot (b \ulcdot c) = (a \ulcdot b) \ulcdot c$.

\item[(iv)] (Distributivity) for $a, b, c \in A$, we have $a \ulcdot (b + c) \subseteq a \ulcdot b + a \ulcdot c$.
\end{enumerate}
\end{Definition}

Concepts like commutativity, grading, and derivation extend to hyper-algebras. Fixing a~hyper-algebra~$A$, we call it commutative if for all $a, b \in A$, we have $a b = b a$. We say that it is graded by a hyper-group~$G$ if $A = \bigoplus_g A_g$ as an $R$-module and $A_g A_{g'} \subseteq \bigcup_{h \in g g'} A_h$. A~hyper-derivation $d$ on~$A$ is an $R$-module endomorphism such that~$d (a \ulcdot b) \subseteq (d a) \ulcdot b + a \ulcdot (d b)$.

Note that for a hyper-algebra $A$, we can define a (left) hyper-module~$M$ as an $R$-module with binary operator $\ulcdot \colon A \times M \ra \mathrm{SubMod}_R(M)$ satisfying the analogue of the axioms in Definition~\ref{def:hyper-algebra}.

\subsection{Hyper-algebras of modular forms}\label{ssec:modular-forms-hyper-algebra}

Recall the various sets of representations defined in Section~\ref{ssec:modular-forms:symplectic-group}. The hyper-algebra of holomorphic Siegel modular forms as a vector space is
\begin{gather*}
 \Mgall= \bigoplus_{\sigma, \rho} \Mg( \sigma \otimes \rho),\qquad \sigma \in \IrrRepRR,\qquad \rho \in \IrrRepAf
\end{gather*}
with product $fg = f \,\ulotimes\, g \subset \Mgall$ defined as
\begin{gather}
 \lspan \Big\{ (\phi_\sigma \otimes \phi_\rho) \circ ( f \otimes g )
\colon \sigma \in \IrrRepRR,\, \phi_\sigma \colon \sigma_f \otimes \sigma_g \ra \sigma,\nonumber\\
\hphantom{\lspan \Big\{ (\phi_\sigma \otimes \phi_\rho) \circ ( f \otimes g )
\colon }{} \rho \in \IrrRepAf,\, \phi_\rho \colon \rho_f \otimes \rho_g \ra \rho
 \Big\}.\label{eq:def:hyper-product}
\end{gather}
Note that both $\IrrRepRR$ and $\IrrRepAf$ are hyper-groups, and the hyper-algebra of Siegel modular forms is graded by both. Our notation for submodules, for example $\Mg (\bullet \otimes \rho)$, has the obvious meaning.

We will also work with the hyper-algebras of almost holomorphic Siegel modular forms
\begin{gather*}
 \aMgall= \bigcup_{0 \le d} \bigoplus_{\sigma, \rho} \aMg{d}( \sigma \otimes \rho ),\qquad \sigma \in \IrrRepRR,\qquad \rho \in \IrrRepAf
\end{gather*}
with product as in~\eqref{eq:def:hyper-product}.

\subsection{Computing the hyper-algebra product in Sage}
In the case of elliptic modular forms, all irreducible weights are $1$-dimensional. Types can, however, be arbitrary dimensional. We adopt an example from~\cite{raum-2017}, which illustrates how to compute products in $\Mgall$ of modular forms with nontrivial type. For convenience, we give Sage code~\cite{sage-6.9}\footnote{Computations make implicit use of the libraries~\cite{flint-2.5.2,pari-2.7.3} and possibly further ones that are less obvious from the source code of Sage.}, which the reader can modify to perform his or her own computation~-- the code is not optimized for performance, but for clarity.

We consider the representation~$\rho_3$, which we realize as a matrix representation by
\begin{gather*}
 \rho_3(T)= \begin{pmatrix}
 1 & 0 & 0 \\
 0 & 0 & 1 \\
 -1 & -1 & -1
 \end{pmatrix} ,\qquad
 \rho_3(S) = \begin{pmatrix}
 0 & 1 & 0 \\
 1 & 0 & 0 \\
 -1 & -1 & -1
 \end{pmatrix}.
\end{gather*}
Its kernel contains the congruence subgroup~$\Gamma(3)$. The space of Eisenstein series of weight~$12$ and type $\rho_3$ has dimension one. A basis element $E_{12,\rho_3}$ can be obtained from vector-valued Hecke operators that are discussed in Section~\ref{ssec:modular-forms:hecke-operators}. We let $\zeta$ be a third root of unity. The image of $\rmT_3$ on the level~$1$ Eisenstein series of weight~$12$ is
\lstset{language=Python, basicstyle=\small\sffamily}
\begin{lstlisting}
    K.<zeta> = CyclotomicField(3)
    R.<q3> = K[[]]

    E12 = EisensteinForms(1,12).basis()[0]
    E12T3 = lambda n: vector([
     	    3** 6 * E12.qexp(n//3+1).subs(q=q3**9).add_bigoh(3*n)
              , 3**-6 * E12.qexp(3*n).subs(q=q3)
              , 3**-6 * E12.qexp(3*n).subs(q=zeta*q3)
              , 3**-6 * E12.qexp(3*n).subs(q=zeta**2*q3)
              ])
\end{lstlisting}
Fourier expansions are computed in terms of $\text{q3} = e(\tau \slash 3)$. The components of an Eisenstein series of type~$\rho_3$ can be found by applying a homomorphism from $\rmT_3 \bbone$ to $\rho_3$, which can be computed by the method that we describe below when decomposing~$\rho_3\rho_3$,
\begin{lstlisting}
    E12rho3 = lambda n: matrix(3, [ 1,-1/3,-1/3,-1/3
                                  , -1/3,1,-1/3,-1/3
                                  , -1/3,-1/3,1,-1/3
                                  ]) * E12T3(n)
\end{lstlisting}

To compute $E_{12,\rho_3} \,\ulotimes\, E_{12,\rho_3}$, we have to decompose the $9$-dimensional representation $\rho_3 \otimes \rho_3$. Since the kernel of $\rho_3$ has finite index in~$\SL{2}(\ZZ)$, character theory for finite groups is one way to achieve this. We use a more general method that applies to all representation: By exhibiting the trivial representation in $( \rho_3 \otimes \rho_3 )^\vee \otimes \rho$ for various representations $\rho$, we compute the space of homomorphisms $\Hom(\rho_3 \otimes \rho_3, \rho)$. For a systematic decomposition of $\rho_3 \otimes \rho_3$, one could apply the MeatAxe algorithm~\cite{holt-eick-obrien-2005} in conjunction with a multimodular approach.

We start by defining the representation matrices of the trivial representation~$\bbone$, $\rho_3$, and $\rho_3 \otimes \rho_3$,
\begin{lstlisting}
    s_triv = t_triv = identity_matrix(QQ, 1)
    s3 = matrix(QQ, 3, [0,1,0, 1,0,0, -1,-1,-1])
    t3 = matrix(QQ, 3, [1,0,0, 0,0,1, -1,-1,-1])
    s33 = s3.tensor_product(s3)
    t33 = t3.tensor_product(t3)
\end{lstlisting}
We compute homomorphisms between representations by employing the isomorphism of vector spaces $\Hom(\rho,\rho') \cong \big(\rho^\vee \otimes \rho'\big) (\bbone)$. Representation matrices of the dual representation $\rho^\vee$ are given by the transpose inverses of those of~$\rho$. We can immediately compute homomorphisms from~$\rho_3\rho_3$ to~$\bbone$ and $\rho_3$,
\begin{lstlisting}
    dual = lambda m: m.transpose().inverse()
    hom = lambda s1,t1,s2,t2: \
            (dual(s1).tensor_product(s2)-1).right_kernel() \
            .intersection((dual(t1).tensor_product(t2)-1).right_kernel())

    hom_triv = hom(s33,t33, s_triv,t_triv)
    hom_rho3 = hom(s33,t33, s3,t3)
\end{lstlisting}
There is one copy of $\bbone$ and two copies of $\rho_3$ in $\rho_3\rho_3$. We are facing the problem of decomposing their complement. In our case, it turns out that it consists of two inequivalent one-dimensional representations with representation matrices
\begin{gather*}
 \rho_\zeta(S)= \begin{pmatrix}
 1
 \end{pmatrix},\qquad
 \rho_\zeta(T) = \begin{pmatrix}
 \zeta
 \end{pmatrix}
\qquad\text{and}\qquad
 \rho_{\zeta^2}(S) = \begin{pmatrix}
 1
 \end{pmatrix},\qquad
 \rho_{\zeta^2}(T)= \begin{pmatrix}
 \zeta^2
 \end{pmatrix}.
\end{gather*}
In Sage, we implement them by means of
\begin{lstlisting}
    K.<zeta> = CyclotomicField(3)

    szeta = identity_matrix(K,1)
    tzeta = matrix(K,1,[zeta])

    szeta2 = identity_matrix(K,1)
    tzeta2 = matrix(K,1,[zeta**2])
\end{lstlisting}

First, we determine the kernel of the homomorphism that we have determined so far
\begin{lstlisting}
    bm = reduce( lambda s,l: s.intersection(matrix([l]).right_kernel())
               ,   [ hom_triv.basis()[0] ]
                 + [b[ix::3] for b in hom_rho3.basis() for ix in range(3)]
               , VectorSpace(QQ,9) ) \
         .basis_matrix().transpose()
    srest = bm.solve_right(s33*bm)
    trest = bm.solve_right(t33*bm)
\end{lstlisting}
Second, we observe that $S$ acts trivially on that kernel, so that it suffices to decompose the action of $T$ into eigenspaces. In the present case they are defined over a third order cyclotomic extension of the rationals. In fact, they are isomorphic to $\rho_\zeta$ and $\rho_{\zeta^2}$ given above
\begin{lstlisting}
    hom_zeta = hom(s33,t33, szeta,tzeta)
    hom_zeta2 = hom(s33,t33, szeta2,tzeta2)
\end{lstlisting}

We construct matrices from the homomorphism spaces that we previously determined
\begin{lstlisting}
    phi_zeta = matrix(hom_zeta.basis()[0])
    phi_zeta2 = matrix(hom_zeta2.basis()[0])
    phi_rho3_1 = matrix([hom_rho3.basis()[0][ix::3] for ix in range(3)])
    phi_rho3_2 = matrix([hom_rho3.basis()[1][ix::3] for ix in range(3)])
\end{lstlisting}
Assembling the results that we have computed via Sage, we find that with respect to the given bases, we have homomorphisms
\begin{alignat*}{4}
& \phi_{\bbone}\colon \ && \rho_3 \otimes \rho_3 \lra \bbone,\qquad && \begin{pmatrix}
 1 & \frac{1}{2} & \frac{1}{2} & \frac{1}{2} & 1 & \frac{1}{2} & \frac{1}{2} & \frac{1}{2} & 1
 \end{pmatrix},& \\
 & \phi_{\zeta}\colon \ && \rho_3 \otimes \rho_3 \lra \rho_{\zeta},\qquad &&
 \begin{pmatrix}
 1 & \zeta + 1 & -\zeta & \zeta + 1 & \zeta & -1 & -\zeta & -1 & -\zeta-1
 \end{pmatrix},&\\
& \phi_{\zeta^2}\colon \ && \rho_3 \otimes \rho_3 \lra \rho_{\zeta^2} ,\qquad && \begin{pmatrix}
 1 & -\zeta & \zeta + 1 & -\zeta & -\zeta - 1 & -1 & \zeta + 1 & -1 & \zeta
 \end{pmatrix},& \\
& \phi_{\rho_3,1}\colon \ && \rho_3 \otimes \rho_3 \lra \rho_3 ,\qquad && \begin{pmatrix}
 1 & 0 & -1 & -1 & -1 & -2 & 0 & 1 & -1 \\
 -1 & 0 & 1 & -1 & 1 & 0 & -2 & -1 & -1 \\
 -1 & -2 & -1 & 1 & -1 & 0 & 0 & -1 & 1
 \end{pmatrix},& \\
 & \phi_{\rho_3,2}\colon \ && \rho_3 \otimes \rho_3 \lra \rho_3,\qquad &&
 \begin{pmatrix}
 0 & 1 & -1 & -1 & 0 & -3 & 1 & 3 & 0 \\
 0 & 1 & 3 & -1 & 0 & 1 & -3 & -1 & 0 \\
 0 & -3 & -1 & 3 & 0 & 1 & 1 & -1 & 0
 \end{pmatrix}.&
\end{alignat*}

Summarizing, we find that $E_{12, \rho_3} \,\ulotimes\, E_{12, \rho_3}$ is supported on $\bbone$, $\rho_\zeta$, $\rho_{\zeta^2}$, and $\rho_3$. The correspon\-ding subspaces are spanned by elements whose Fourier coefficients are too large to display them all. We confine ourselves to the trivial type:
\begin{gather*}
 \frac{564856947200}{1594323} - \frac{1894333004462080000}{84584326707}q - \frac{1261863434802833408000}{28194775569}q^2 + O\big(q^3\big)\\
 \qquad{} \in E_{12, \rho_3} \,\ulotimes\, E_{12, \rho_3} \cap \rmM\big( 24 \otimes \bbone \big).
\end{gather*}
Beyond this, we illustrate how to compute the remaining Fourier expansions with Sage. The tensor square of~$E_{12, \rho_3}$ with precision at least~$n$ is given by the following function~$\text{Esq}(n)$
\begin{lstlisting}
    Esq = lambda n: vector(R,[c1*c2 for c1 in E12rho3(n)
                                    for c2 in E12rho3(n)])

    print phi_triv * Esq(3)
    print phi_zeta * Esq(3)
    print phi_zeta2 * Esq(3)
    print (phi_rho3_1 * Esq(3), phi_rho3_2 * Esq(3))
\end{lstlisting}

\subsection{From weights to isomorphism classes of irreducible weights}

We have defined weights as complex, finite-dimensional representations of $\KgRR$ without any further restriction. If we wished to construct a graded algebra of modular forms whose grading includes all these weights, this would not be possible. All representations of~$\KgRR$ together do not constitute a set, but rather they are objects in a category $\Rep\big(\KgRR\big)$. Given this fact, one might be inclined to pass from weights to isomorphism classes of weights~-- the skeleton of $\Rep\big(\KgRR\big)$. Alternatively, we can and will focus on a fixed set~$\ovRep\big(\KgRR\big)$ of representatives of isomorphism class. For every $\sigma \in \ovRep\big(\KgRR\big)$ the space of modular forms~$\Mg(\sigma)$ of weight~$\sigma$ is a vector space.

Let us consider to what extent we can define a multiplication. Given $f \in \Mg( \sigma_f )$ and $g \in \Mg( \sigma_g )$ with $\sigma_f, \sigma_g \in \ovRep\big(\KgRR\big)$, we find that
\begin{gather*}
 fg = f \otimes g \in \Mg ( \sigma_f \otimes \sigma_g )
\end{gather*}
for their tensor product. Its weight, in general, is not in $\ovRep\big(\KgRR\big)$. Denote by $\ov{\sigma_f \sigma_g}$ the corresponding representative. A priori, $fg$ lies in $\rmC^\infty( \HSg \ra V(\sigma_f\sigma_g) )$. We have to obtain from it an element in $\rmC^\infty\big( \HSg \ra V(\ov{\sigma_f\sigma_g}) \big)$. The most naive way is to choose an endomorphism $\phi$ in $\Hom( \sigma_f\sigma_g, \ov{\sigma_f\sigma_g})$ and compose it with $fg$. The resulting function $\phi \circ fg$ depends on $\phi$. More precisely, it depends on $f$ and $g$ up to automorphisms of $\ov{\sigma_f\sigma_g}$.

Except in the case $g=1$, there seems to be no natural choice of endomorphisms for all tensor products. For instance, $\sym^{l}\sym^{l'}$ contains $\sym^{l+l'}$ in a natural way by realizing $\sym^l$ and $\sym^{l'}$ as a representation on polynomials of degree~$l$ and $l'$ in $g$ variables. But, for example if $g = 2$, there is no natural inclusion of $\sym^{l+l'-2}$ into $\sym^{l}\sym^{l'}$ for $l,l' \ge 2$. The situation becomes even more difficult for $g \ge 3$, because then irreducible representations can occur with multiplicities greater than~$1$ in tensor products of irreducible representations.

Since there is no natural choice of a single homomorphism from $\sigma_f\sigma_g$ to $\ov{\sigma_f\sigma_g}$, we resort to taking the span over all of them. The product of $f$ and $g$ in this setting is defined as the space
\begin{gather*}
 \lspan \big\{ \phi \circ fg \colon \phi \in \Hom\big( \sigma_f\sigma_g,\, \ov{\sigma_f\sigma_g} \big)
 \big\} \subset \bigoplus_{\sigma \in \ovRep\big(\KgRR\big)} \Mg( \sigma).
\end{gather*}
This product yields a hyper-algebra of modular forms whose weights run through all isomorphism classes of representations. Since at this point, we have already lost the algebra structure and arrived at hyper-algebras, there is no additional harm in considering only irreducible weights. This motivates our definition in Section~\ref{ssec:modular-forms-hyper-algebra}.

\subsection{Covariant differential operators on hyper-algebras}\label{ssec:hyper-algebras:differential-operators}

As in the case of tensor products the image of lowering and raising operators has, in general, reducible weight. For example, in the case $g=2$ and if $l \ge 2$, the weight~$\rmR \det^k\sym^l = \det^k\sym^2\sym^l$ allows for a decomposition into irreducible weights $\det^{k-2}\sym^{l-2}$, $\det^{k-1}\sym^l$, and $\det^k\sym^{l+2}$. It is therefore natural to define a vector space valued action of covariant differential operators.

Let $\CC [ \rmL, \rmR ]$ be the polynomial algebra in two formal variables $\rmL$ and $\rmR$. Notation overlaps with the one for lowering and raising operators, but it will be clear from the context to what we refer when writing $\rmL f$ or $\rmR f$. The action of $\rmL$ and $\rmR$ on $\aMgall$ is given by
\begin{gather*}
 \rmL f = \lspan \big\{
 \phi \circ \rmL_{\sigma_f} f \colon \sigma \in \IrrRepRR,\, \phi \colon \sigma \ra \rmL \sigma_f \big\}\qquad\text{and}\\
 \rmR f = \lspan \big\{
 \phi \circ \rmR_{\sigma_f} f \colon \sigma \in \IrrRepRR,\, \phi \colon \sigma \ra \rmR \sigma_f \big\}.
\end{gather*}

Viewing the differential operators $\rmL$ and $\rmR$ jointly as a vector-valued Hecke operator at the infinite place acting on almost holomorphic Siegel modular forms, we set for $f \in \aMgall$
\begin{gather*}
 \bbT_\infty= \CC [ \rmT_\infty ],\qquad \rmT_\infty f = \rmL f + \rmR f \subset \aMgall.
\end{gather*}
One readily verifies by employing the defining formulas of lowering and raising operators in~\cite{klemm-poretschkin-schimannek-raum-2015} that $\rmT_\infty$ acts as a hyper-derivation.

\subsection{Hecke actions on hyper-algebras}\label{ssec:hyper-algebras:hecke-operators}

Based on the vector-valued Hecke operators that we have introduced in Section~\ref{ssec:modular-forms:hecke-operators}, we get an additional hyper-module structure on (almost) holomorphic Siegel modular forms. Let
\begin{gather*}
 \bbT_\rmf= \CC [ \rmT_p \colon p \ \text{prime}]
\end{gather*}
be the polynomial ring in infinitely many formal variables $\rmT_p$. As in the case of covariant differential operators, notation coincides with the one for actual Hecke operators, but this should not lead to confusion. The action of $\rmT_p$ on $\aMgall$ is defined by
\begin{gather*}
 \rmT_p f = \lspan \big\{ \phi \circ \rmT_p f \colon \rho \in \IrrRepAf,\, \phi \colon \rmT_p \rho_f \ra \rho \big\} \subset \aMgall.
\end{gather*}
This equips $\aMgall$ with the structure of a $\bbT_\rmf$-hyper-module.

\subsection{The formal Hecke algebra}\label{ssec:hyper-algebras:hecke-algebra}

We combine both formal Hecke algebras into one
\begin{gather*}
 \bbT= \bbT_{\infty}\bbT_{\Af} = \CC [ \rmT_\infty, \rmT_p \colon p\ \text{prime}].
\end{gather*}
Recall that $\bbT_\infty \subset \bbT$ acts on almost holomorphic Siegel modular forms by hyper-derivations, and~$\bbT_\rmf$ acts by endomorphisms.

\section{Essential surjectivity of tensor products of modular forms}\label{sec:structure-results}

The hyper-algebra structure and the action of the formal Hecke algebra make it rather easy to formulate essential surjectivity for products of modular forms. The blueprint for such results are the ones in~\cite{klemm-poretschkin-schimannek-raum-2015, raum-2017}. The first of them, phrased in the language that we have developed reads as follows. For all $k \ge 8$, all $l,k-l \ge 4$, and all congruence representations~$\rho$, we have
\begin{gather}\label{eq:surjectivity:genus-1-finite-places}
 \rmM^{(1)}( k \otimes \bullet )= \rmE^{(1)}( k \otimes \bullet)+ \bbT_\rmf E^{(1)}(l) \,\ulotimes\, \bbT_\rmf E^{(1)}(k-l).
\end{gather}
It is now straightforward to ask for analogues.

\subsection{A conjecture in the case of genus~2}\label{ssec:structure-results:genus-2}

We start with a conjecture, that we will not prove in this paper. There is experimental evidence in~\cite{raum-2010a} that
\begin{gather}
 \rmM^{(2)}( k \otimes \bbone)= \rmS\rmK^{(2)} ( k \otimes \bbone )+ \sum_{4 \le l \le k - 4}
 \rmS\rmK^{(2)} ( l \otimes \bbone )\cdot \rmS\rmK^{(2)} ( k-l \otimes \bbone),
\end{gather}
where $\rmS\rmK$ denotes the space of (holomorphic) Saito--Kurokawa lifts. In light of this and of formula~\eqref{eq:surjectivity:genus-1-finite-places} an extension to all~$\rho$ seems possible. Could it be that for sufficiently large~$k$,~$l$, and~$k-l$, and for any congruence type~$\rho$ of genus~$2$ Siegel modular forms, we have
\begin{gather}\label{eq:structure-results:genus-2}
 \rmM^{(2)} ( k \otimes \rho )= \rmS\rmK^{(2)} ( k \otimes \rho )+ \bbT_\rmf \rmS\rmK^{(2)} ( l \otimes \bbone )\, \ulotimes\,
 \bbT_\rmf \rmS\rmK^{(2)} ( k-l \otimes\bbone )\quad\text{?}
\end{gather}
Note that we have restricted to scalar weights in the above. To cover vector-valued weights, we have to first study the effect of covariant differential operators on tensor products.

On inspection of the proofs in~\cite{raum-2017}, we observe that applying vector-valued Hecke operators on the right-hand side of~\eqref{eq:structure-results:genus-2} reduces the proof to a certain unfolding of Petersson scalar products and a nonvanishing of twisted $\rmL$-values. In the case of Siegel modular forms, unfolding to twisted spinor $\rmL$-series was studied in~\cite{krieg-raum-2009} following older ideas from~\cite{kohnen-skoruppa-1989}. A sufficient nonvanishing result for the purpose of~\cite{raum-2017} follows from Waldspurger's treatment of half-integral weight modular forms~\cite{kohnen-zagier-1984, waldspurger-1981}. Since the untwisted analogue in the case of genus~$2$ Siegel modular forms has now been established~\cite{furusawa-morimoto-2017}, there is hope to prove the following, stronger statement:
\begin{gather*}
 \rmM^{(2)} ( k \otimes \rho )= \rmS\rmK^{(2)} ( k \otimes \rho )+
 \bbT_\rmf \rmE^{(2)} ( l \otimes \bbone ) \,\ulotimes\, \bbT_\rmf \rmS\rmK^{(2)}( k-l \otimes \bbone)\quad\text{?}
\end{gather*}

\subsection{Hecke operators at all places in genus~1}

Making use of not only Hecke operators at the finite places, but of the full formal Hecke algebra~$\bbT$, one can strengthen surjectivity results such as the one in~\eqref{eq:surjectivity:genus-1-finite-places}.
\begin{Theorem}\label{thm:products-for-differential-operators}Let $l,l' \ge 4$ be even integers, and $\rho$ a congruence type. Then for every even $k \ge l + l'$ we have
\begin{gather}\label{eq:surjectivity:genus-1-infinite-place}
 \rmS^{(1)} ( k \otimes \rho ) \subset \bbT_\infty \big( \bbT E^{(1)}(l) \,\ulotimes\, \bbT E^{(1)}(l') \big).
\end{gather}
\end{Theorem}
\begin{proof}To ease notation, we suppress the superscript $(1)$ for Eisenstein series throughout this proof. The notation in this proof is adopted from~\cite{raum-2017}: We let $\bbone_{|D|}$ be the square of the Kronecker character~$\epsilon_D$ for a negative fundamental discriminant~$D$. Note hat $\bbone_{|D|}$ is a trivial non-primitive Dirichlet character. The modular form $f_{\epsilon_D}$ is the twist of $f$ by $\epsilon_D$. The induction of modular forms for $\Gamma_0(N)$ to vector-valued modular forms for $\SL{2}(\ZZ)$ is denoted by~$\Ind$. The induction of a Dirichlet character~$\chi$ of modulus~$N$ from $\Gamma_0(N)$ to~$\SL{2}(\ZZ)$ is denoted by~$\rho_\chi$.

Without loss of generality, we can assume that $l \le l'$. In complete analogy with~\cite{raum-2017}, we note that the right-hand side of~\eqref{eq:surjectivity:genus-1-infinite-place} is a $\bbT_{\rmf}$-module. It therefore suffices to show the following: Any scalar multiple of a newform~$f$ for the congruence subgroup $\Gamma_0(N)$ vanishes, if for $t = (l + l' - k) \slash 2 \in \ZZ_{\ge 0}$ and all~$D$, the (holomorphic) modular form $f_{\epsilon_D}$ is orthogonal to the almost holomorphic modular form
\begin{gather*}
 \Ind \big( E_{l,\bbone_{|D|}} \big) \otimes \rmR^t \big( E_{l',\bbone_{|D|},\infty} \big).
\end{gather*}

We define the following two vector-valued Eisenstein series:
\begin{gather*}
 E_{l,\bbone_{|D|}} (\tau)= \sum_{n = 1}^\infty \sigma_{l-1, \bbone_{|D|}}(n) q^n,\\
 E_{l,\bbone_{|D|},\infty,s}= \sum_{\gamma \in \Gamma_\infty \backslash \SL{2}(\ZZ)} \frake_{\Gamma_0(N)} \otimes \frake_{\Gamma_0(N)} y^s
 \Big|_{k,\rho_{\bbone_{|D|}} \otimes \rho_{\bbone_{|D|}}} \gamma.
\end{gather*}
Real-analytic vector-valued Eisenstein series have, for example, appeared in~\cite{taylor-2006}.

We establish the described vanishing condition by relating the Petersson scalar product to special $L$-values, as in~\cite{imamoglu-kohnen-2005,kohnen-zagier-1984,raum-2017}. Combining regularization and unfolding as is described in detail in~\cite{raum-2017}, we find that
\begin{gather}
 \big\langle \Ind f_{\epsilon_D}, \Ind \big( E_{l,\bbone_{|D|}} \big) \otimes \rmR^t \big( E_{l',\bbone_{|D|},\infty,s} \big) \big\rangle\nonumber\\
\qquad{} = \int_{\Ga{1} \backslash \HS^{(1)}} \pi \big( f \cdot E_{l,\bbone_{|D|}} (\tau) \cdot \frake_{\Gamma_0(N)} \otimes \frake_{\Gamma_0(N)} \rmR^t y^s \big) \frac{{\rm d} x {\rm d} y}{y^{2-k}},\label{eq:surjectivity:genus-1-infinite-place:scalar-product}
\end{gather}
where $\pi$ is the projection adjoint to the inclusion
\begin{gather*}
 \bbone \lra \rho_{|D|} \otimes \big( \rho_{|D|} \otimes \rho_{N|D|^2} \big) \otimes \rho_{N |D|^2},\qquad
 1 \lmto \sum_{\substack{\gamma \colon \Gamma_0(|D|) \backslash \SL{2}(\ZZ) \\
 \gamma' \colon \Gamma_0(|D| N) \backslash \SL{2}(\ZZ)}}
 \frake_\gamma \otimes \big( \frake_\gamma \otimes \frake_{\gamma'} \big) \otimes \frake_{\gamma'}.
\end{gather*}
We will next evaluate the integral, employ analytic continuation, and then insert $s = 0$, to obtain a product of special values of Dirichlet series.

A direct verification shows that $\lim\limits_{s \ra 0} \rmR^t_{l'} y^{s} = (l')^\uparrow_t y^{-t}$, where $(a)^\uparrow_n = a (a+1) \cdots (a+n-1)$ is the upper factorial. The above scalar product~\eqref{eq:surjectivity:genus-1-infinite-place:scalar-product} therefore equals
\begin{gather*}
 (l')^\uparrow_t \left( \sum_{n=0}^\infty \int_0^\infty c(f_{\epsilon_D};n) e(n {\rm i} y) \sigma_{l-1,\bbone_{|D|}}(n) e(n {\rm i} y)
 y^{-t+s} \frac{{\rm d} y}{y^{2-k}} \right)_{s = 0}\\
\qquad{} = (l')^\uparrow_t \left( \frac{\Gamma(k-1-t)}{(4 \pi)^{k-1-t}} \sum_{n=0}^\infty \frac{\sigma_l(n) c(f_{\epsilon_{D}};n)}
 {n^{k-1-t+s}} \right)_{s = 0}.
\end{gather*}
By the extension of Rankin's result~\cite{rankin-1952} in~\cite{raum-2017}, we find that it does converge absolutely, if $l > l'$ or allows for a suitable analytic continuation of $l = l'$. It can be expressed in terms of special $L$-values as follows
\begin{gather*}
 \frac{L\big( f_{\epsilon_D}, k-1-t \big) L\big( f_{\epsilon_{D}} \times \bbone_{|D|}, k-l-t \big)}
 {L\big( \bbone_{|D|} \bbone_{|D|} \bbone, k - 1 - 2t - l)\big)}.
\end{gather*}
The denominator equals $L\big( \bbone_{|D|}, l' - 1)\big)$ by the relation $k = 2t + l + l'$. The first factor in the numerator is a special value of an Euler product, since $t = ( k - l - l' ) \slash 2 < \frac{k}{2} - 2$ and thus $k - 1 - t > \frac{k}{2} + 1$. In order to inspect the second factor in the numerator, note that $l + t \le \frac{k}{2}$, since we have assumed that $l \le l'$. If $l + t < \frac{k}{2}$ then $k - t - l \ge \frac{k}{2} + 1$, so that $L(f_{\epsilon_D}, k - t - l)$ is the special value of a convergent Euler product. Otherwise, we infer that the central value $L\big(f_{\epsilon_D}, \frac{k}{2}\big)$ vanishes for all negative fundamental discriminants~$D$. Using Waldspurger's and Kohnen--Zagier's results~\cite{kohnen-zagier-1984, waldspurger-1981}, we can argue as in~\cite{raum-2017} to finish the proof. More precisely, the vanishing of~$L\big(f_{\epsilon_D}, \frac{k}{2}\big)$ implies the vanishing of the Shintani lift of~$f$, because $f$ is a scalar multiple of a newform. Since the Shintani lift is injective, we find that~$f = 0$ as desired.
\end{proof}

Theorem~\ref{thm:products-for-differential-operators} does not cover the case of small weights~$k$. The next statement clarifies that $k < l + l'$ cannot appear on the right-hand side of~\eqref{eq:surjectivity:genus-1-infinite-place}.
\begin{Proposition}\label{prop:surjectivity-fails-small-weight}
Suppose that $k < l_1 + l_2$ for positive even $l_1,l_2 \ge 4$ and positive even~$k$. Then any weight~$k$ cusp form~$f$ is orthogonal with respect to the Petersson scalar product to $\rmR^{t_1} g_1 \cdot \rmR^{t_2} g_2$ for $g_1$ and $g_2$ modular forms of weight~$l_1$ and $l_2$, and $t_1,t_2 \in \ZZ_{\ge 0}$ such that $k = l_1 + 2 t_1 + l_2 + 2 t_2$.
\end{Proposition}
\begin{proof}The almost holomorphic modular form $\rmR^{t_1} g_1 \cdot \rmR^{t_2} g_2$ has depth $t_1 + t_2$. That is, it allows for a decomposition
\begin{gather*}
 \rmR^{t_1} g_1 \cdot \rmR^{t_2} g_2= \sum_{t = 0}^{t_1 + t_2} \rmR^t h_t
\end{gather*}
for holomorphic modular forms~$h_t$ of weight $k - 2t$. Since $d > t_1 + t_2$, each term in this sum is orthogonal to $f$ by Shimura's orthogonality relations~\cite{shimura-1987}~-- also confer~\cite{pitale-saha-schmidt-2015b}.
\end{proof}

\section{Adelic automorphic representations}\label{sec:automorphic-representations}

Automorphic representation theory in most modern settings focuses on adelic representations. That is, one investigates the right regular representation of $\rmG(\bbA)$ on $L^2( \rmG({\mathbb Q})\backslash \rmG(\bbA))$, where $\rmG = \PGSp{g}$. It is an important aspect of the theory that one can split any automorphic representation into a restricted tensor product of local components. The local theory is over $p$-adic fields and the theory over the infinite places.

In the classical theory of modular forms these two aspects of automorphic representation theory are reflected by Hecke operators and covariant differential operators. To support our claim that hyper-algebras of modular forms together with the action of $\bbT$ are mitigating between the classical language and the representation theoretic one, we show how to pass from classical modular forms to the attached local representations and how some of their aspects can be interpreted in terms of hyper-algebras. In particular, we rephrase our result in Section~\ref{sec:structure-results} in terms of tensor products of automorphic representations.

\subsection{Harish-Chandra modules}\label{ssec:automorphic-representations:harish-chandra}

It is formally correct to work with the group $G = \PGSp{g}$ in this subsection. However, as is common, we will instead work with $G = \Sp{g}$, i.e., $G(\RR) = \Sp{g}(\RR)$ for which all aspects that we discuss here are the same. Further, we adopt notation from~\cite{raum-2015a} to shorten this exposition.

Given any almost holomorphic Siegel modular form~$f$, we can associate to it an irreducible Harish-Chandra module. Recall that a Harish-Chandra module is an admissible~\gKmodule. A \gKmodule\ is a representation of $\rmK = \KgRR$ that is simultaneously a $\frakg$-module with $\frakg = \fraksp{g}$, and for which these two structures are compatible. A detailed definition can be found in~\cite[Section~3.3.1]{wallach-1988}. A \gKmodule~$M$ is called admissible if it is a unitarizable~$K$-representation and if $\Hom_K(\sigma, M)$ is finite-dimensional for every finite-dimensional $K$-representation~$\sigma$.

Starting with~$f$, we produce a function~$\rmA_\infty(f)$ on~$G$, $\rmA_\infty$ stands for adelization at the infinite place, which here is isomorphic to~$\RR$:
\begin{gather*}
 \rmA_\infty(f)(g)= \big( f \big|_{\sigma, \rho} g \big) \big(iI^{(g)}\big).
\end{gather*}
It takes values in $V(\sigma) \otimes V(\rho)$, and when contracting with $V(\sigma)^\vee$, we obtain a space of functions that take values in~$V(\rho)$ and which is a $K$-module isomorphic to $\sigma^\vee \otimes V(\rho)$. We denote the contraction by $\rmA_\infty(f) \cdot V(\sigma)^\vee$, and the \gKmodule\ generated by it will be denoted by
\begin{gather*}
 \ov{\rmA}_\infty(f)= (\frakg,\rmK) \big( \rmA_\infty(f) \cdot V(\sigma)^\vee \big).
\end{gather*}

Let $\frakk \subset \frakg$ be the Lie algebra of $K$. It was established in~\cite{raum-2015a} that application of covariant differential operators and the action of the $\frakk$-complement $\frakm \subseteq \frakg$ commute with passing back and forth between modular forms and $K$-types in Harish-Chandra modules. Concretely, we established commutativity of the following diagram, featuring the hyper-derivations $\rmL$ and~$\rmR$:
\begin{center}
\begin{tikzpicture}
\matrix(m)[matrix of math nodes,
column sep = 10em, row sep = 3em,
text height = 1.5em, text depth = 1.25ex]
{ f & \rmA_\infty(f) \\
 \rmR f & \frakm^+ \rmA_\infty(f), \\
};

\path
(m-1-1) edge[|->] node[above] {$\rmA_\infty$} (m-1-2)
(m-2-1) edge[->] node[above] {$\rmA_\infty$} (m-2-2)

(m-1-1) edge[|->] node[left] {$\rmR$} (m-2-1)
(m-1-2) edge[|->] node[right] {$\frakm^+$} (m-2-2);
\end{tikzpicture}%
\hspace*{4em}
\begin{tikzpicture}
\matrix(m)[matrix of math nodes,
column sep = 10em, row sep = 3em,
text height = 1.5em, text depth = 1.25ex]
{ f & \rmA_\infty(f) \\
 \rmL f & \frakm^- \rmA_\infty(f). \\
};

\path
(m-1-1) edge[|->] node[above] {$\rmA_\infty$} (m-1-2)
(m-2-1) edge[->] node[above] {$\rmA_\infty$} (m-2-2)

(m-1-1) edge[|->] node[left] {$\rmL$} (m-2-1)
(m-1-2) edge[|->] node[right] {$\frakm^-$} (m-2-2);
\end{tikzpicture}
\end{center}

The product of vector-valued modular forms then contains information about the tensor product of Harish-Chandra modules. And thus it yields a~lower bound on smooth functions in the tensor product of $G(\RR)$-representations. Given a Harish-Chandra module we can obtain from it a~representation of~$\rmG(\RR)$. In our case, we denote this representation by $\rmG(\RR) \ov{A}_\infty$.
\begin{Corollary}\label{cor:tensor-product-at-infinity}
Given two almost holomorphic Siegel modular forms~$f$ and $g$, we have the following inclusion of \gKmodules:
\begin{gather*}
 \ov{\rmA}_\infty \big( \bbT_\infty f \,\ulotimes\, \bbT_\infty g \big)\subseteq \ov{\rmA}_\infty (f)\otimes \ov{\rmA}_\infty (g).
\end{gather*}

By passing back to representations of $\Sp{g}(\RR)$, we find that
\begin{gather*}
 \rmG(\RR) \ov{\rmA}_\infty \big( \bbT_\infty f \,\ulotimes\, \bbT_\infty g \big) \subseteq \rmG(\RR)\ov{\rmA}_\infty (f) \otimes \rmG(\RR)\ov{\rmA}_\infty (g).
\end{gather*}
\end{Corollary}
\begin{proof}This follows from the commutative diagrams above, except that we have to verify that the hyper-product $\bbT_\infty f \,\ulotimes\, \bbT_\infty g$ yields an admissible \gKmodule. This becomes clear when inspecting the action of the center of $\frakk$.
\end{proof}

\subsection{The finite places}\label{ssec:automorphic-representations:finite-places}

As opposed to the infinite place it is important to insist on the group $\rmG = \PGSp{g}$ when treating the finite places. Recall that as a maximal compact group over $\Qp$, we choose $\KgQp = \rmG(\Zp)$. Also recall that any type in this paper is a congruence type, which means that it gives rise to a~representation over~$\Af$.

To approximate at $p$, we need the relation $\rmG(\Qp) = \rmG({\mathbb Q}) \KgQp$. Consider a vector-valued almost holomorphic Siegel modular form~$f$ of type $\rho = \otimes' \rho_p$. Given $g k \in \rmG(\Qp)$ we set
\begin{gather*}
 \rmA_p(f)(gk)= \rho_p^{-1}(k) \big( f \big|_{\sigma,\rho} g \big)\big( i I^{(g)} \big).
\end{gather*}
As a first remark note that $f$ is constant on $\ker \rho_p \subseteq \KgQp$. The contraction of $\rmA_p(f)$ with $V(\rho_p)^\vee$ will be denoted by $\ov{\rmA}_p(f)$. It is a space of functions on $\rmG(\Qp)$ taking values in $V(\sigma) \otimes V(\otimes'_{p \ne p'} \rho_{p'})$. As a $\KgQp$-module it has isomorphism type $\rho_p^\vee \otimes \big( V(\sigma) \otimes V(\otimes'_{p \ne p'} \rho_{p'}) \big)$.

The commutative diagram corresponding to the one in Section~\ref{ssec:automorphic-representations:harish-chandra} is
\begin{center}
\begin{tikzpicture}
\matrix(m)[matrix of math nodes,
column sep = 10em, row sep = 3em,
text height = 1.5em, text depth = 1.25ex]
{ f & \rmA_p(f) \\
 \rmT_p f & \Delta_p \rmA_p(f), \\
};

\path
(m-1-1) edge[|->] node[above] {$\rmA_p$} (m-1-2)
(m-2-1) edge[->] node[above] {$\rmA_p$} (m-2-2)

(m-1-1) edge[|->] node[left] {$\rmT_p$} (m-2-1)
(m-1-2) edge[|->] node[right] {$\Delta_p$} (m-2-2);
\end{tikzpicture}
\end{center}
where $\Delta_p$ is the vector space with basis consisting of determinant $p$ matrices in $\KgQp$. In analogy with the infinite case, we write $\rmG(\Qp) \ov{\rmA}_p(f)$ for the representation over~$\Qp$ that corresponds to~$f$.
\begin{Corollary}\label{cor:tensor-product-at-p}Given Siegel modular forms~$f$ and~$g$, we have the following inclusion of $\rmG(\Qp)$-representations:
\begin{gather*}
 \rmG(\Qp) \rmA_p \big( \bbT_p f \,\ulotimes\, \bbT_p g \big) \subseteq \rmG(\Qp) \rmA_p (f) \otimes \rmG(\Qp)\rmA_p (g).
\end{gather*}
\end{Corollary}

\subsection{Global tensor products}\label{ssec:automorphic-representations:global}

We now transfer the statements of Corollaries~\ref{cor:tensor-product-at-infinity} and~\ref{cor:tensor-product-at-p} to the global setting.
\begin{Theorem}\label{thm:global-automorphic-products}Given almost holomorphic Siegel modular forms~$f$ and~$g$, with automorphic representations $\varpi(f)$ and $\varpi(g)$ associated with them. Suppose that for a newform $h$ we have
\begin{gather*}
 \Ind(h)\in \bbT f \,\ulotimes\, \bbT g.
\end{gather*}
Then the associated automorphic representation~$\varpi(h)$ is contained in the tensor product of those associated with~$f$ and $g$:
\begin{gather*}
 \varpi(h)\lhra \varpi(f) \otimes \varpi(g).
\end{gather*}
\end{Theorem}
\begin{proof}Using strong approximation, we obtain a map from holomorphic Siegel modular forms~$f$ to functions
\begin{gather*}
 \rmA(f)= \rmA_\infty(f) \otimes \Big( \bigotimes_p \rmA_p(f) \Big)
\end{gather*}
on $\rmG(\bbA)$, which generate an automorphic representation~$\varpi(f) = \varpi_\infty(f) \otimes \otimes'_p \varpi_p(f)$ associated with~$f$. Note that the Harish-Chandra module attached to a representation of $\rmG(\RR)$ consists of the subspace of smooth vectors. In particular, it is a subspace of $V(\varpi_\infty(f))$. The commutative diagrams in Sections~\ref{ssec:automorphic-representations:harish-chandra} and~\ref{ssec:automorphic-representations:finite-places} show that $\rmA_\infty ( \bbT_\infty f ) \subset V(\varpi_\infty)$ and $\rmA_p ( \bbT_p f ) \subset V(\varpi_p)$.
\end{proof}

As a consequence of Theorem~\ref{thm:global-automorphic-products}, we can reinterpret~\eqref{eq:surjectivity:genus-1-finite-places} and~\eqref{eq:surjectivity:genus-1-infinite-place} in terms of tensor products of global representations. This provides a proof of Theorem~\ref{thm:main-theorem:automorphic}.

\subsection*{Acknowledgements}

The author was partially supported by Vetenskapsr{\aa}det Grant~2015-04139. The author is grateful to the referees for various helpful comments.

\pdfbookmark[1]{References}{ref}
\LastPageEnding

\end{document}